\documentclass[11pt,a4paper]{amsart}

\usepackage{amsmath}
\usepackage{amssymb}

\theoremstyle{plain}   
\begingroup 

\newtheorem{theorem}{Theorem}[section]   
\newtheorem*{th_HW}{Theorem HW}
\newtheorem*{th_DYS}{Theorem DYS}
\newtheorem*{th_I}{Theorem I}
\newtheorem{corollary}[theorem]{Corollary}     
\newtheorem{lemma}[theorem]{Lemma}         
\newtheorem{proposition}[theorem]{Proposition}  
\newtheorem{fact}[theorem]{Fact}
\endgroup

\theoremstyle{definition}
\newtheorem{definition}[theorem]{Definition}   

\theoremstyle{remark}
\newtheorem{remark}[theorem]{Remark}        

\numberwithin{equation}{section}

\newcommand{\ep}{\varepsilon}

\newcommand{\R}{{\mathbb R}}
\newcommand{\N}{{\mathbb N}}

\newcommand{\A}{{\cal A}}
\newcommand{\at}{\widetilde{\mathcal A}}
\newcommand{\ct}{\widetilde{\mathcal C}}

\newcommand{\vf}{\varphi}

\newcommand{\diam}{\operatorname{diam}}

\newcommand{\fu}{\varphi}
\newcommand{\C}{{\cal C}}

\newcommand{\I}{{\cal I}}
\newcommand{\cal}{\mathcal}

\usepackage[dvips]{color}

\begin{document}

\title[Denjoy-Young-Saks theorem]{Properties of Hadamard directional derivatives: Denjoy-Young-Saks theorem for functions on Banach spaces}

\maketitle

\centerline{\bf Lud\v ek Zaj\'\i\v cek}
\centerline{\it Charles University, Faculty of Mathematics and Physics,}
\centerline{\it Sokolovsk\'a 83, 186 75 Praha 8, Czech Republic.}
\centerline{\it  zajicek@karlin.mff.cuni.cz}

\bigskip
\bigskip

The classical Denjoy-Young-Saks theorem on Dini derivatives of arbitrary functions $f: \R \to \R$
 was extended by U.S. Haslam-Jones (1932) and A.J. Ward (1935) to arbitrary functions on $\R^2$.
 This extension gives the strongest relation  among upper and lower Hadamard directional derivatives
	$f^+_H (x,v)$, $f^-_H (x,v)$ ($v \in X$) which holds almost everywhere for an arbitrary function 
	$f:\R^2\to \R$. Our main result extends the theorem of Haslam-Jones and Ward to functions on separable Banach spaces.

\bigskip

 
 {\it Keywords}:\ \ Hadamard upper and lower directional derivatives, Denjoy-Young-Saks theorem, separable Banach space, Hadamard differentiability, Fr\' echet differentiability, Hadamard subdifferentiability,  Fr\' echet subdifferentiability, $\tilde{\C}$-null set, $\Gamma$-null set, Aronszajn null set
\bigskip

{\it 2000 Mathematics Subject Classification}:\  Primary: 46G05; Secondary: 26B05

\bigskip
\bigskip

\section{Introduction}
The classical Denjoy-Young-Saks theorem gives the strongest relation between Dini derivatives $D^+f$, $D_+f$, $D^-f$, $D_-f$, which holds for an arbitrary functions almost everywhere. It reads as
 follows.
\begin{th_DYS}
Let $f$ be an arbitrary real function on $\R$. Then, at almost all $x\in \R$, one from the following
 assertions holds:
\begin{enumerate}
\item  $f'(x) \in \R$ exists,
\item  $D^+f(x) = D^-f(x)=\infty$\ \ and\ \  $D_+f(x)=D_-f(x)=-\infty$,
\item $D^+f(x) =\infty$, $D_-f(x)=-\infty$, \ \ and \ \ $D_+f(x)=D^-f(x) \in \R$,
 \item $D^-f(x) =\infty$, $D_+f(x)=-\infty$, \ \ and \ \ $D_-f(x)=D^+f(x) \in \R$.
 \end{enumerate}
\end{th_DYS}

A natural generalization of Denjoy-Young-Saks theorem to measurable functions on $\R^2$
was proved by U.S. Haslam-Jones \cite{Ha} and was extended to arbitrary functions on $\R^2$
 by A.J. Ward \cite{Wa1}. This generalization works with upper and lower directional derivatives,
 however it is not possible to use the usual (``Dini'') directional derivatives (even for continuous functions, see \cite{Be}), but it is necessary to work with  Hadamard upper and lower directional derivatives  $f^+_H(x,v)$ and $f^-_H(x,v)$ (see Subsection 2.2 below for definitions) which
 were defined (possibly first time) in \cite{Ha} under the name ``directed upper and lower derivatives''. (Haslam-Jones considered their definition as the
 fundamental idea of the article, see \cite[p. 121]{Ha}.) 

The theorem of Haslam-Jones and Ward (see  \cite[13. Summary, p. 31]{Ha} and \cite[Theorem V, p. 352]{Wa1}) can be reformulated in the following way.

	\begin{th_HW}
Let $f$ be an arbitrary real function on $X=\R^2$. Then, at almost all $x\in X$, one from the following
 assertions holds:
\begin{enumerate}
\item  $f$ is Hadamard differentiable at $x$, 
\item $f^+_H (x,v)= \infty$ for each $v \in X$  and  $f^+_H (x,v)= -\infty$ for each $v \in X$,
\item  $f^+_H (x,v)= \infty$ for each $v \in X$ and $L(v): =f^-_H (x,v) \in X^*$,
\item $f^-_H (x,v)= -\infty$ for each $v \in X$ and $L(v): =f^+_H (x,v) \in X^*$.	
 \end{enumerate} 
	\end{th_HW}
	We show (Theorem \ref{dys2}) that this result holds in any separable Banach space $X$ if
	``at almost all points'' means ``at all points except for a set from $\tilde C$''.
	(Note that each set from the $\sigma$-ideal $ \tilde C$ is both Haar null and $\Gamma$-null.)
	
	We emphasize that the condition (i) above holds (in an arbitrary Banach space $X$) if and only if
	 $f^+_H (x,v) =   f^-_H (x,v)$ for each $v\in X$ and $L(v): =f^+_H (x,v) \in X^*$, and so
	 Theorem HW and Theorem \ref{dys2} give a relation among Hadamard directional derivatives
	$f^+_H (x,v)$, $f^-_H (x,v)$ ($v \in X$) which holds a.e. for an arbitrary function on $X$.
	 Moreover (see Remark \ref{nejs}) it is the strongest relation of this type.
	
	S. Saks (\cite[Theorem 14.2]{Sa2}) inferred Theorem HW from results (due to F. Roger) on the contingents
	 (= tangent Bouligand cones) of arbitrary subsets in $\R^3$. His formulation of Theorem HW is 
	 formally very different from that of \cite{Ha} and \cite{Wa1}; in particular he did not used explicitly ``directed derivatives''. Further, his formulation works with Fr\' echet differentiability
	 and  Fr\' echet subgradients (supergradients). So his formulation of Theorem HW is not equivalent 
	 to that of \cite{Ha} and \cite{Wa1} in infinite-dimensional spaces (although it is easily equivalent in $\R^n$). However, we show (Theorem \ref{dys1}) that also this Saks' ``stronger formulation'' extends to some infinite-dimensional spaces. Namely,
	it is  true, if ``almost 
	 all'' is taken in the sense of $\Gamma$-null sets and $X$ is a separable Banach space, in which
	 J. Lindenstrauss and  D. Preiss \cite{LP} proved the `` $\Gamma$-a.e. Rademacher theorem'' for Fr\' echet
	 differentiability of real functions; e.g.  if
	
	\begin{enumerate}
	\item $X$ is a subspace of $c_0$, or
	\item $X=C(K)$, where $K$ is a separable compact space, or
	\item  $X$ is the Tsirelson space.
	\end{enumerate}
	
A.J. Ward (\cite[p. 347]{Wa1}) claims that his arguments can be used to prove Theorem HW in $\R^n$,
 but ``to avoid confusion of quasi-geometrical detail, however, this extension has not been made''
 and it seems that no proof of this extension can be found in the literature.
 However, it is well-known
for a long time that the following interesting special cases of  Theorem HW hold in $\R^n$:
\begin{equation}\label{poli}
\text{Every $f$ on $\R^n$ is differentiable at almost all points at which it is Lipschitz.}
\end{equation}
\begin{equation}\label{mono}
\text{Every cone-monotone function on $\R^n$ is almost everywhere differentiable.}
\end{equation}
Indeed, \eqref{poli} is the well-known Stepanov theorem  and \eqref{mono}
 was proved in \cite{CC}.

For generalizations (using Hadamard differentiability and $\tilde{\C}$-null sets) of \eqref{poli} and \eqref{mono} to separable Banach spaces 
 see \cite[Theorem 3.4
]{ZaGH} and \cite[Corollary 17]{Du}.
 
Generalizations (using Fr\' echet differentiability and $\Gamma$-null sets) of \eqref{poli} and \eqref{mono} to $c_0$ (and similar spaces) were observed in \cite{ZaCM}.

Note also that an incomplete version (see Remark \ref{inco} below) of Theorem HW (which does not imply the linearity of $L$ in
Theorem HW (iii),(iv)) was proved by S. Saks (\cite[Th. 7, p. 238]{Sa1}) before \cite{Wa1}. A similar  incomplete version (see Remark \ref{inco} below) of our Theorem \ref{dys2} follows from a result of \cite{Pr}. Note that D. Preiss
 in \cite{Pr} works with another $\sigma$-ideal of null sets, which is smaller (and possibly strictly
 smaller) than the $\sigma$-ideal $\tilde{\C}$.

Finally note that the present paper is a continuation of author's articles \cite{ZaGH},
\cite{Zadgh}, \cite{ZaCM} which were motivated by the article \cite{Io} by A.D. Ioffe, in which
 the relations among directional Hadamard derivatives of arbitrary functions on separable Banach spaces are investigated.
 A result of \cite{Io} (see Theorem I from Preliminaries) is one from ingredients of our proof.

The article is organized as follows. In Section 2 we recall definitions of the main notions
 and some their properties. The core of the article is  the first step of the proof
 of Lemma \ref{ct} of Section 3, which is the only technical part of the article. This proof
 is an essential modification of the proof of \cite[Theorem 15]{Du}.

The proof of Section 4 is essentially well-known. Indeed, it is ``one-half'' of Mal\' y's   proof
 from  \cite{Ma}, which shows that each Rademacher theorem (for real functions) easily implies a corresponding Stepanov theorem.

Section 5 contains the main results. Their proofs follow easily from results of Sections 3 and 4
 and the above mentioned Ioffe's result.

\section{Preliminaries}

\subsection{Basic notation}

In the following, by a Banach space we mean a real Banach space. If $X$ is a Banach space, we set
    $S_X:= \{x \in X: \|x\|=1\}$. The symbol $B(x,r)$ will denote the open ball with center $x$ and radius $r$. The characteristic function of a set $C\subset X$ is denoted by $\chi_C$.
    
      Let $X$ be a Banach space, $x \in X$, $v \in S_X$ and $\delta >0$. Then we define the open cone
     $C(x,v,\delta)$ as the set of all $y \neq x$ for which  $\|v - \frac{y-x}{\|y-x\|}\| < \delta$.

The following easy inequality is well known (see e.g.\ \cite[Lemma~5.1]{MS}):

\begin{equation}\label{triangle}
\text{if }u,w\in X\setminus\{0\},\text{ then }
 \left\|\frac{u}{\|u\|}-\frac{w}{\|w\|}\right\|
\leq \frac{2}{\|u\|}\, \|u-w\|.
\end{equation}

\subsection{Directional derivatives, derivatives and subgradients}

 In this subsection we suppose that $f$ is a real function defined on an open subset $G$  of a Banach space $X$.
  
  We say  that {\it $f$ is Lipschitz at $x \in G$} if  $\limsup_{y \to x} \frac{|f(y)-f(x)|}{\|y-x\|} < \infty$.

  The {\it Hadamard directional and one-sided  directional derivatives} 
    of $f$ at $x\in G$ in the direction $v\in X$ are defined respectively by
   $$f'_H(x,v) := \lim_{z \to v, t \to 0} \frac{f(x+tz)-f(x)}{t}\ \ \text{and}\ \  f'_{H+}(x,v) := \lim_{z \to v, t \to 0+} \frac{f(x+tz)-f(x)}{t}.$$
  
   Following \cite{Io}, we denote the {\it upper  and lower Hadamard one-sided directional derivatives} 
    of $f$ at $x$ in the direction $v$ by
   $$f^+_H(x,v):=\limsup_{z \to v, t\to 0+}(f(x+tz)-f(x))t^{-1} \ \ \text{and}\ \ f^-_H(x,v):=\liminf_{z \to v, t\to 0+}(f(x+tz)-f(x))t^{-1} .$$
   
   Obviously, 
   \begin{equation}\label{hom}
   \text{if $v \neq 0 $ and $ \lambda >0$, then  $f^+_H(x,\lambda v) = \lambda f^+_H(x,v)$ and $f^-_H(x,\lambda v) = \lambda f^-_H(x,v)$.}
   \end{equation}

   Further (see \cite[the proof of Proposition 4.4]{Pe}), the following easy fact holds.
   \begin{lemma}\label{hdvn}
   The following assertions are equivalent.
   \begin{enumerate}
   \item
   $\liminf_{y \to x} \frac{f(y)-f(x)}{\|y-x\|} > -\infty$,
   \item
   $f^-_H(x,0)  > -\infty$,
  \item
  $f^-_H(x,0) =0$.
  \end{enumerate}
  \end{lemma}
   
    
 The usual modern definition of the Hadamard derivative is the following:
 
 A functional $L \in X^*$ is said to be a {\it Hadamard derivative} of $f$ at a point $x \in X$ if
  $$ \lim_{t \to 0} \frac{f(x+tv)-f(x)}{t} = L(v)\ \ \  \text{for each}\ \ \ v \in X$$
   and the limit is uniform with respect to $v \in C$, whenever $C\subset X$ is a compact set.
    In this case we set $f'_H(x) := L$.
    
    It is well-known (see \cite{Sha}, cf. \cite[Exercise 1, p. 132]{Pe}) that  $L\in X^*$ is a  Hadamard derivative of $f$ at $x$ if and only if
      $f'_H(x,v)= L(v)$ for each $v \in X$.

    Recall (see, e.g. \cite{Sha} or \cite{Pe}), that if  $f$ is Hadamard differentiable at $a \in G$, then $f$ is G\^ ateaux
     differentiable at $a$, and 
    if $f$ is locally Lipschitz on $G$, then also the opposite implication holds.
 
     Further, if  $f$ is Fr\' echet differentiable at $a \in G$, then $f$ is Hadamard
     differentiable at $a$, and  
    if $X= \R^n$, then also the opposite implication holds.

     A functional $L \in X^*$ (see, e.g. \cite{ACL}) is said to be a {\it Fr\' echet subgradient} (resp. {\it Hadamard subgradient}) of $f$ at $x$, if
         $$ \liminf_{h \to 0}  \frac{f(x+h) - f(x) - L(h))}{ \|h\|} \geq 0\ \ \ \ (\text{resp.}\ \  L(h) \leq f^-_H(x,v), \ v \in X).$$
A functional $L \in X^*$  is said to be a {\it Fr\' echet supergradient} (resp. {\it Hadamard supergradient}) of $f$ at $x$, if $-L$ is a	{\it Fr\' echet subgradient} (resp. {\it Hadamard subgradient}) of $-f$ at $x$.

We say that $f$ is  Fr\' echet (resp. Hadamard) subdifferentiable at $x$ if there exists a  Fr\' echet (resp. Hadamard) subgradient of $f$ at $x$. Fr\' echet (resp. Hadamard) superdifferentiability is defined analogously.
We will need the following well-known facts.
\begin{fact}\label{ssdi}
\begin{enumerate}
\item
 If $L \in X^*$ is  a Fr\' echet subgradient of $f$ at $a \in G$, then $L$ is  a Hadamard
     subgradient of $f$  at $a$, and
    if $X= \R^n$, then also the opposite implication holds.
		\item
$f$ is  Fr\' echet (resp. Hadamard) differentiable at $x$ if and only if $f$ is both 
 Fr\' echet (resp. Hadamard) subdifferentiable at $x$ and  Fr\' echet (resp. Hadamard) superdifferentiable at $x$.
\end{enumerate}
\end{fact}
\begin{proof}
For (i) see, e.g., \cite[p. 266, (4.8) and Proposition 4.6)]{Pe}. For (ii) see, e.g., \cite[Remark on p. 266 and Corollary 4.5]{Pe}.
\end{proof}

     \subsection{Null sets}

     \begin{definition}\label{atil}
Let $X$ be a separable Banach space and $B \subset X$ a Borel set.
\begin{enumerate}
\item [(i)]
 If $0\neq v \in X$  and $\ep > 0$, then we say that 
$B \in \at(v, \ep)$ if
$\{t: \fu (t) \in B \}$ is Lebesgue null whenever
  $\fu : \R \to X$ is
 such that the function $x \mapsto \fu(x) - xv$ has Lipschitz constant
 at most $\ep$.
\item[(ii)]
We say that  $B \in \ct$ if $B$ can be written in the form
 $B = \bigcup_{n=1}^{\infty} B_n$, where $B_n \in \at (v_n,\ep_n)$
 for some $v_n \neq 0$ and $\ep_n >0$.
\end{enumerate}
We define $\ct = \ct(X)$ as the system of all $A \subset X$ for which there exists a Borel set $B \in \ct$ such that $A \subset B$. The sets from $\ct$ will be called $\ct$-null.
 \end{definition}
\begin{remark}\label{kct}
\begin{enumerate}
\item[(a)]
It is easy to see that $\ct$ is $\sigma$-ideal.
\item[(b)]
The present definition of the system $\ct$  slightly differs from that of \cite{PZ}, under which each member of $\ct$ is Borel.
 So $\ct$ by our definition is the  $\sigma$-ideal generated by the corresponding system of \cite{PZ}.
 \item[(c)]
 Each member of $\ct$ is contained in an Aronszajn null set (see \cite[Proposition 13]{PZ}) and thus it is Haar null.
  (For definition of Aronszajn and Haar null sets see \cite{BL}.) Moreover, each member of $\ct$ is also $\Gamma$-null
   (see \cite{ZaMB}). The important $\sigma$-ideal of $\Gamma$-null subsets of $X$ was introduced
	 in \cite{LP}.
	 \item[(d)]
	 The main theorem  of \cite{PZ} works with the system $\at$. We will not need its definition; note
	 only that (obviously) $\at \subset \ct$ and it is not known whether $\at(X) = \ct(X)$ for each $X$.
\end{enumerate}
\end{remark}

    \begin{definition}
    Let $X$ be a Banach space.
		\begin{enumerate}
		\item
		We say that $A \subset X$ is porous at a point $x \in X$ if there exits a sequence $x_n \to x$ and $c>0$ such that $B(x_n, c \|x_n-x\|) \cap A = \emptyset$. 
		If, moreover, $(x_n)$ can be taken so that all $x_n$ belong to a half-line $\{x+tv: t\geq 0\}$
		 ($v\neq 0$), we say that $A$ is directionally porous at $x$.
		\item
	    We say that $A \subset X$ is porous (resp. directionally porous) if $A$ is porous (resp. directionally porous) at each point $x \in A$. 
			\item
            We say that $A \subset X$ is $\sigma$-porous (resp. $\sigma$-directionally porous) if it is a countable union of porous (resp. directionally porous) 
       sets.
			\end{enumerate}
       \end{definition}
If $X$  is a separable Banach space, then  (see, e.g., \cite[3.1]{ZaAAA})    
 \begin{equation}\label{dpct}
\text{each $\sigma$-directionally porous subset of $X$ belongs to $\at$, and so to $\ct$.}
\end{equation}

   We will need the following immediate consequence of \cite[Theorem 3.1(b)]{Io}. It uses
	 the notion of sparse sets (i.e., sets which can be covered by countably many Lipschitz hypersurfaces). We will not need the formal definition of sparse sets, but only the almost obvious
	 fact (see, e.g., \cite{Io}), that
	\begin{equation}\label{spdp}
	\text{each sparse set is $\sigma$-directionally porous.}
	\end{equation}
	\begin{th_I}\label{io}
	Let $X$ be a separable space, $G \subset X$  an open set and $f:G \to \R$  an arbitrary function. 
		Let $H$ be the set of all $x \in G$,
	for which
   $f_H^{-}(x,v) +  f_H^{-}(x,v) > 0$ or $f_H^{+}(x,v) +  f_H^{+}(x,v) < 0$ for some $v \in X$.
	 Then $H$ is sparse.
\end{th_I}

\section{Lower Lipschitzness via cone lower Lipschitzness}

We will need the following easy lemma.
\begin{lemma}\label{spar}
Let $X$ be a separable space, $G \subset X$  an open set and $f:G \to \R$  an arbitrary function. 
Denote by $T(f)$ the set of all $x \in G$ for which there exists a cone $C_x=C(x,v,\delta)$ such that
$$  f(x) < \liminf_{y \to x, y \in C_x}\ f(y).$$
Then $T(f) \in \tilde{\C}$.
\end{lemma}
\begin{proof}
Let $(r_n)_1^{\infty}$  be a sequence of all rational numbers. Set
$$  T(f,n):= \{x\in T(f):\ f(x) < r_n <\liminf_{y \to x, y \in C_x}\ f(y)\}.$$
Obviously,  $T(f) = \bigcup_{n=1}^{\infty} T(f,n)$. Fix an arbitrary $n \in \N$ and $x \in T(f,n)$.
Then there exists  $\eta >0$  such that $f(y) > r_n$ for each $y \in C_x \cap B(x,\eta)$.
So $y \notin T(f,n)$ for each  $y \in C_x \cap B(x,\eta)$. This clearly implies that
 $T(f,n)$ is directionally porous and so $T(f)$ is $\sigma$-directionally porous. Thus
$T(f) \in \tilde{\C}$ by \eqref{dpct}. 
\end{proof}
(Note that the above proof shows that $T(f)$ is sparse, but we will not need this stronger fact.)

\begin{lemma}\label{ct}
Let $X$ be a separable Banach space, $G \subset X$  an open set, $v \in S_X$, $0< \delta<1$ and $f:G \to \R$  an arbitrary function. Denote by $A(f,v,\delta)$ the set of all $x \in G$, for which
\begin{enumerate}
\item  \ \ \  $f(y) \geq f(x)$ whenever  $y \in C(x,v, \delta) \cap B(x,\delta)$ and
\item \ \ \ $\liminf_{z \to x}  \frac{f(z)-f(x)}{\|z-x\|} = - \infty.$ 
\end{enumerate}
Then  $A(f,v,\delta) \in \tilde{\C}$.
\end{lemma}
\begin{proof}
The set  $A(f,v,\delta)$ need not be Borel. From this reason we will proceed in two steps. 
 In the first (essential) step we will prove $A(f,v,\delta) \in \tilde{\C}$ under 
the assumption
 that  $A(f,v,\delta)$ is Borel. In the second step we will reduce the general case to the case
 when $f$ is lower semi-continuous, in which  $A(f,v,\delta)$ is Borel.

\underline{First step}:\  \ Suppose that  $B:= A(f,v,\delta)$ is Borel.
It is sufficient to prove that $B \in \tilde{\A}(v, \ep)$, where $\ep:= \delta/5$.
 So let $\vf: \R \to X$ be a mapping such that $\psi: t \to \vf(t)-tv$ has Lipschitz constant at most $\ep$.
 Suppose to the contrary that $\lambda(\vf^{-1}(B))>0$. Choose an open interval $I \subset \R$ such that
 \begin{equation}\label{hv}
 \diam I < \frac{\delta}{1+\ep}\ \ \ \text{and}\ \ \  \lambda(B^*)>0,
 \end{equation}
 where  $B^*:=  \vf^{-1}(B) \cap I$. 
 
 We will show that the function  $g:= f \circ \vf\restriction_{B^*}$ is nondecreasing. So suppose
  that $t_1, t_2 \in B^*$ and $t_1<t_2$. Since  $\vf(t_2) - \vf(t_1) = (t_2 - t_1) v + (\psi(t_2) - \psi(t_1))$, using 
   \eqref{triangle} with $u:= (t_2 - t_1) v$ and $w:= \vf(t_2) - \vf(t_1)$, we obtain
   $$ \left\| v - \frac{\vf(t_2) - \vf(t_1)}{\|\vf(t_2) - \vf(t_1)\|}\right\| \leq \frac{2}{|t_2-t_1|} \|\psi(t_2)-\psi(t_1)\|\leq 2 \ep < \frac{\delta}{2}.$$
   So  $\vf(t_2) \in C (\vf(t_1), v, \delta/2)$.  Further, using \eqref{hv}, we obtain
    \begin{equation}\label{prf}
    (1-\ep) (t_2-t_1) \leq \|\vf(t_2)- \vf(t_1)\| \leq (1+\ep) (t_2-t_1) < \delta.
        \end{equation}
    So we obtain
        \begin{equation}\label{tjd}
    \vf(t_2) \in C (\vf(t_1), v, \delta/2) \cap B(\vf(t_1), \delta).
    \end{equation}
     Since $\vf(t_1) \in B$, by (i) we obtain  $f(\vf(t_2)) \geq f(\vf(t_1))$. So $g$ is nondecreasing.
     
     Set  $J:= (\inf B^*, \sup B^*)$ and $\tilde{g}(t):= \sup\{g(\tau):\ \tau \in B^* \cap (-\infty,t]\}$ for $t \in J$.
      Then $\tilde{g}$  extends $g$  and is finite and nondecreasing on $J$. 
       Consequently we can choose $t_0 \in B^* \cap J$ such that
       \begin{equation}\label{dh}
       (\tilde{g})'(t_0) \in \R\ \ \ \text{and}\ \ \ t_0\ \ \ \text{is a density point of}\ \ \  B^*.
       \end{equation}
       
       Consequently for $K:= | (\tilde{g})'(t_0)|+1$ there exists  $0< \alpha < \ep$ such that for each $0<h< \alpha$
        there exists a number $t \in (t_0-h- \ep h, t_0-h) \cap B^*$ such that
        \begin{equation}\label{HVE}
        |f(\vf(t)) - f(\vf(t_0))| \leq K |t-t_0|.
        \end{equation}
        Set $x_0:= \vf(t_0)$ and $L:= 6 K(1+\ep) \delta^{-1}$.
        
        Since $\vf(t_0) \in B$, by (ii) we can find $y \in X$ such that
        \begin{equation}\label{oy}
        h:= 6 \delta^{-1}\|y-x_0\|<\alpha\ \ \ \text{and}\ \ \  \frac{f(y)-f(x_0)}{\|y-x_0\|} < -L.
        \end{equation}
        Since $0<h<\alpha$, we can choose $t \in (t_0-h- \ep h, t_0-h) \cap B^*$ such that \eqref{HVE} holds.
       
       Since $t, t_0 \in B^*$, we can use \eqref{prf} and \eqref{tjd} (with $t_1:=t$ and $t_2:=t_0$) and obtain
       \begin{equation}\label{prf2}
       (1-\ep) h \leq \|x_0- \vf(t)\| \leq (1+\ep)^2 h
       \end{equation}
       and
        \begin{equation}\label{txn}
       \left\|  \frac{x_0-\vf(t)}{\|x_0-\vf(t)\|}-v     \right\| < \frac{\delta}{2}.
       \end{equation}
       Using \eqref{triangle}, \eqref{prf2}  and $\ep< 1/3$, we obtain
       $$ \left\|  \frac{x_0-\vf(t)}{\|x_0-\vf(t)\|} - \frac{y-\vf(t)}{\|y-\vf(t)\|}    \right\|
        \leq  2 \frac{\|y-x_0\|}{\|x_0-\vf(t)\|} \leq \frac{2\|y-x_0\|}{(1-\ep)6 \delta^{-1}\|y-x_0\|}=
         \frac{\delta}{3(1-\ep)} < \frac{\delta}{2}.$$
         Consequently \eqref{txn} implies 
         $$ \left\| \frac{y-\vf(t)}{\|y-\vf(t)\|} - v\right\| < \delta.$$
          Further  \eqref{oy} and \eqref{prf2} imply
          $$  \|y-\vf(t)\| \leq \|x_0-\vf(t)\| +h  < 5h < 5 \ep = \delta.$$
          Thus  $y \in C(\vf(t),v,\delta) \cap B(\vf(t),\delta)$, and since $\vf(t) \in B$, we obtain by (i)
           that  $f(y) \geq f(\vf(t))$. Consequently, using \eqref{oy}, we obtain
           $$ f(x_0) - f(\vf(t)) = (f(x_0) -f(y)) + (f(y)- f(\vf(t))) \geq f(x_0)-f(y) > L \|y-x_0\|.$$
           On the other hand, \eqref{HVE} gives
           $$ f(x_0)- f(\vf(t)) \leq K |t-t_0| < K(1+\ep) h = K (1+\ep) 6 \delta^{-1} \|y-x_0\| =L \|y-x_0\|,$$
            a contradiction.
						
						\underline{Second step}:\ \ Now consider the general case. First observe that we can suppose that $f$ is bounded. Indeed,
						 setting  $f_n:= \max(-n, \min(f,n))$ for $n \in \N$, it is easy to see
						 that  $A(f,v,\delta) \subset \bigcup_{n=1}^{\infty} A(f_n,v,\delta)$.
						
						Set 
						$$  g(x):= \min( f(x), \liminf_{t \to x} f(t)).$$
						Obviously, $g$ is a bounded lower semi-continuous function. The definitions of
						$A(f,v,\delta)$ and $g$ clearly imply that
						\begin{equation}\label{troj}
						A(f,v,\delta) \subset  A(g,v,\delta)  \cup  \{x \in A(f,v,\delta):\ f(x) > g(x)\}.
						\end{equation}
						Now we will show that  $A(g,v,\delta)$ is a Borel set. To this end first consider the set
						$$ M:= \{x \in X:\ g(y) \geq g(x)\ \ \text{whenever}\ \ y \in C(x,v,\delta) \cap B(x,\delta)\}.$$
						We will show that $M$ is closed. To this end consider points $x_n\in M$, $n\in \N$, such
						 that $x_n \to x$. Let $y\in C(x,v,\delta) \cap B(x,\delta)$. It is easy to see that
						 there exists $n_0$ such that, for all $n\geq n_0$,  we have $y\in C(x_n,v,\delta) \cap B(x_n,\delta)$, and consequently  $g(x_n) \leq g(y)$. Since $g$ is lower-semicontinuous, we obtain
 $g(x) \leq g(y)$. So $x\in M$.

Further denote
$$  L:= \{x \in X:\ \liminf_{y \to x}  \frac{g(y)-g(x)}{\|y-x\|} > - \infty\}  \ \ \text{and}$$ 
  	$$L_n:= \{x\in X:\  g(y)-g(x) \geq -n \|y-x\|\ \ \text{if}\ \ \|y-x\| < 1/n\}, \ n \in \N.$$					
		Then clearly  $L= \bigcup_{n=1}^{\infty} L_n$. We will show that each $L_n$ is closed. To this end consider points $x_k\in L_n$, $k\in \N$, such
						 that $x_k \to x$. Let $y \in B(x,1/n)$. It is easy to see that
						 there exists $k_0$ such that, for all $k\geq k_0$,  we have $\|y- x_k\|< 1/n$, and consequently  $g(y)-g(x_k)		 \geq -n \|y-x_k\|$. Consequently
	\begin{multline*}
	g(y)-g(x) \geq g(y) - \liminf_{k \to \infty} g(x_k) = \limsup_{k \to \infty}  (g(y)-g(x_k))\\
	 \geq \limsup_{k \to \infty} (-n\|y-x_k\|) = -n\|y-x\|,
								\end{multline*}						
		and so $x \in L_n$. 
							
		Thus we obtain that $A(g,v,\delta)= M \setminus L$ is  Borel. Using the first step of the
							proof, we obtain   $A(g,v,\delta) \in \tilde C$.
							
							Further consider $ x \in A(f,v,\delta)$ such that  $ f(x) > g(x)$.  Since $f(y) \geq f(x)$ for each
							 $y \in C(x,v,\delta) \cap B(x,\delta)$, we clearly have $g(y) \geq f(x)$ for each
							 $y \in C(x,v,\delta) \cap B(x,\delta)$. So $x \in T(g)$, where $T(g)$ is as in Lemma \ref{spar}. 
							  Therefore $\{x \in A(f,v,\delta):\ f(x) > g(x)\} \subset T(g)$. So, using Lemma \ref{spar} and \eqref{troj},
							   we obtain  $ A(f,v,\delta) \in \ct$.
						           \end{proof}
											
		\begin{lemma}\label{lm}
		Let $X$ be a separable Banach space, $G \subset X$  an open set, $v \in S_X$, $0< \delta<1$, $K>0$ and $f:G \to \R$  an arbitrary function. 
		Let $M(f,v,\delta,K)$ be the set of all $x \in G$,
	for which
	\begin{enumerate}
\item  \ \ \  $f(y) - f(x) \geq -K \|y-x\|$ whenever  $y \in C(x,v, \delta) \cap B(x,\delta)$ and
\item \ \ \ $\liminf_{z \to x}  \frac{f(z)-f(x)}{\|z-x\|} = - \infty.$ 
\end{enumerate}
	Then  $M(f,v,\delta,K)\in \tilde \C$.										
	\end{lemma}									
	\begin{proof}
	Choose  $x^* \in X^*$ with  $x^*(v) > 2K$ and set $g:= f+x^*$, \  $\tilde{\delta}:= \min(\delta, K \|x^*\|^{-1})$. We will show that 	$M(f,v,\delta,K)\subset A(g,v, \tilde{\delta})$, where	 $A(g,v, \tilde{\delta})$ is as in Lemma \ref{ct}. To this end choose an arbitrary $x \in M(f,v,\delta,K)$
	 and  $y \in C(x,v, \tilde{\delta}) \cap B(x, \tilde{\delta})$. Since  $\tilde{\delta} \leq \delta$,
	 we have
	\begin{equation}\label{prne}
	f(y)- f(x) \geq -K \|y-x\|.
	\end{equation}
	We have $\|v- (y-x)\|y-x\|^{-1}\| < \tilde{\delta}$, and therefore
	$$ \left| x^*(v) - \frac{x^*(y)-x^*(x)}{\|y-x\|}\right| < \tilde{\delta}\|x^*\|\leq K.$$
	Consequently
	\begin{equation}\label{drne}
	x^*(y)-x^*(x) \geq x^*(v)\|y-x\| - K\|y-x\| > K\|y-x\|.
	\end{equation}
	Adding \eqref{prne} and \eqref{drne}, we obtain  $g(y) \geq g(x)$.
	Since clearly  $\liminf_{z \to x}  \frac{g(z)-g(x)}{\|z-x\|} = - \infty$,
	 we obtain  $x  \in  A(g,v, \tilde{\delta})$. Thus Lemma \ref{ct} implies $M(f,v,\delta,K)\in \tilde \C$.	
	\end{proof}
		\begin{lemma}\label{hl}
		Let $X$ be a separable Banach space, $G \subset X$  an open set and $f:G \to \R$  an arbitrary function. 
		Let $P_1$ be the set of all $x \in G$,
	for which
	\begin{enumerate}
\item  \ \ \   $f_H^{-}(x,v)> -\infty$ for some $v \in X$ and
\item \ \ \ $\liminf_{z \to x}  \frac{f(z)-f(x)}{\|z-x\|} = - \infty.$ 
\end{enumerate}
	Then  $P_1\in \tilde \C$.										
	\end{lemma}
	\begin{proof}
	  Let $\{v_n: n \in \N\}$ be a dense subset of $S_X$.
		By Lemma \ref{lm} it is sufficient to show (under the notation of  Lemma \ref{lm}) that
		\begin{equation}\label{pb}
		P_1 \subset  \bigcup_{n,m,k=1}^{\infty} M(f,v_n, 1/m, k).
		\end{equation}
		To prove \eqref{pb}, consider an arbitrary $x \in P$. Choose $v \in X$ such that
		$f_H^{-}(x,v)> -\infty$. By \eqref{hom} and Lemma \ref{hdvn} we can (and will) suppose that
		 $v \in S_X$. Choose
		 $k \in \N$
		 such that  $f_H^{-}(x,v)> -k$. We can clearly find $\delta>0$ such that
		$f(y)-f(x) > -k\|y-x\|$ for each $y \in C(x,v,\delta) \cap B(x,\delta)$. Further
		 choose $n\in \N$ such that $\|v_n-v\|< \delta$ and $m \in \N$ such that
		 $\|v_n-v\| + 1/m < \delta$. It is easy to see that $x \in M(f,v_n, 1/m, k)$.
		\end{proof} 
	Applying Lemma \ref{hl} to $-f$, we obtain
	\begin{corollary}\label{hl2}
		Let $X$ be a separable Banach space, $G \subset X$  an open set and $f:G \to \R$  an arbitrary function. 
		Let $P_2$ be the set of all $x \in G$,
	for which
	\begin{enumerate}
\item  \ \ \   $f_H^{+}(x,v)< \infty$ for some $v \in S_X$ and
\item \ \ \ $\limsup_{z \to x}  \frac{f(z)-f(x)}{\|z-x\|} =  \infty.$ 
\end{enumerate}
	Then  $P_2\in \tilde \C$.		
		\end{corollary}
										
\section{Subdifferentiability via lower Lipschitzness}				
					
		\begin{proposition}\label{dlsd}
		Let $X$ be a separable Banach space and $\I$ a $\sigma$-ideal of subsets of $X$. Suppose that
   \begin{enumerate}
  \item [(R)]\ \ Every  real Lipschitz function on each open $\emptyset \neq H \subset X$ is Fr\' echet (resp. Hadamard) differentiable  on $H$ except for a set from $\I$.
    \end{enumerate}
 Let $G \subset X$ be an open set, $f$ an arbitrary real function on $G$. Let
  $A$ be the set of all $ x \in G$ such that 
	 $\liminf_{z \to x}  \frac{f(z)-f(x)}{\|z-x\|} > - \infty$ and
$f$ is not Fr\' echet (resp. Hadamard) subdifferentiable at $x$.
 Then  $A \in \I$.
  \end{proposition}
  \begin{proof}
  For each  $n \in \N$, set
  $$  L_n: = \{x \in G:\ f(y)-f(x) \geq  -n \|y-x\|\ \ \text{whenever}\ \ \|y-x\| \leq 1/n\}.$$
	 Let $(U^n_k)_{k \in \N}$ be a cover of $X$ with open sets  with diameter smaller than $1/n$   and set  $A^n_k : = A \cap L_n \cap U^n_k$. Since clearly 
    $A = \bigcup_{n,k=1}^{\infty}\, A^n_k$, it is sufficient to prove that each  $A^n_k$ belongs to $\I$.
		
		So suppose that natural $n$, $k$ with $A^n_k\neq \emptyset$ are fixed.   Set 
         $$ h_{n,k}(x):= \sup \{h(x): h\ \text{is Lipschitz with constant}\ n \ \text{on} \  U^n_k \ \text{and}\ 
          h \leq f\ \text{on}\  U^n_k\}.$$
          Consider, for $a \in A^n_k$, the function
          $$ h_a(x):= f(a)- n \|x-a\|,\ x \in  U^n_k.$$
           Clearly each $h_a$ is Lipschitz with constant $n$   and, since $A^n_k \subset L_n$ and $\diam  U^n_k < 1/n$, we have $h_a \leq f$ on $ U^n_k $.  Consequently we obtain that $ h_{n,k}$ is finite and Lipschitz with constant $n$ on $U^n_k$,
            \begin{equation}\label{mezi}
            h_{n,k}  \leq f \ \ \text{on}\ \ \ \ U^n_k \ \ \ \text{and}\ \ \  \ \ h_{n,k}(a) = f(a)
           \ \ \text{for each}\ \ a \in A^n_k .
                        \end{equation}
																By (R) there exists $N^n_k \in \I$ such that  $h_{n,k}$ is Fr\' echet (resp. Hadamard)
			differentiable at each point of $U^n_k \setminus N^n_k$.	Using \eqref{mezi}, we easily obtain
			 (see, e.g., \cite[Proposition 4.7
			]{Pe}) that $f$ is  Fr\' echet (resp. Hadamard)  subdifferentiable at 
			 each point of $A^n_k \setminus N^n_k$. Therefore  $(A^n_k \setminus N^n_k) \cap A^n_k = 
			\emptyset$. Consequently $A^n_k \subset N^n_k$, and so $A^n_k \in \I$. 
          \end{proof}
			\begin{corollary}\label{s1}
					Let $X$ be a Banach space such that
					\begin{enumerate}
\item[(a)] $X^*$  is separable, and
\item[(b)] each porous subset of $X$ is $\Gamma$-null.
\end{enumerate}
 Let $G \subset X$ be an open set and $f$ an arbitrary real function on $G$. Let
  $S_1$ be the set of all $ x \in G$ such that either
	\begin{enumerate}
	\item
	 $\liminf_{z \to x}  \frac{f(z)-f(x)}{\|z-x\|} > - \infty$ and
$f$ is not Fr\' echet  subdifferentiable at $x$ or
\item
$\limsup_{z \to x}  \frac{f(z)-f(x)}{\|z-x\|} < \infty$ and
$f$ is not Fr\' echet  superdifferentiable at $x$.
\end{enumerate}
 Then $S_1$ is a $\Gamma$-null set. 
		\end{corollary}
	\begin{proof}
					Let $\I$ be the $\sigma$-ideal of $\Gamma$-null sets in $X$.
					Then the condition (R) of Proposition \ref{dlsd}  for Fr\' echet differentiability
					 holds (see \cite{LP} or \cite{LPT}).
					Applying 
					Proposition \ref{dlsd} to $f$ and $-f$, we obtain $\Gamma$-null sets $A_1$ and $A_2$, respectively. So $S_1 = A_1 \cup A_2$ is $\Gamma$-null.
					\end{proof}
			\begin{corollary}\label{s2}
			Let $X$ be a separable Banach space.  Let $G \subset X$ be an open set and $f$ an arbitrary real function on $G$. Let
  $S_2$ be the set of all $ x \in G$ such that either
	\begin{enumerate}
	\item
	 $\liminf_{z \to x}  \frac{f(z)-f(x)}{\|z-x\|} > - \infty$ and
$f$ is not Hadamard   subdifferentiable at $x$ or
\item
$\limsup_{z \to x}  \frac{f(z)-f(x)}{\|z-x\|} < \infty$ and
$f$ is not Hadamard   superdifferentiable at $x$.
\end{enumerate}
	 Then $S_2 \in\tilde{\C}$. 
		\end{corollary}
	\begin{proof}
	Let  $\I := \tilde{\C}$. Then \cite[Theorem 12 and Proposition 13]{PZ} imply that the condition (R)  of Proposition \ref{dlsd} for G\^ ateaux
					differentiability holds. However, since (R) deals with Lipschitz functions, it holds
					 for Hadamard differentiability as well (see, e.g., \cite[Proposition 2.25]{Pe}).
	Applying 
					Proposition \ref{dlsd} to $f$ and $-f$, we obtain  sets $A_1\in \tilde{\C}$ and $A_2\in \tilde{\C}$, respectively. So $S_2 = A_1 \cup A_2 \in \tilde{\C}$.
	\end{proof}

					\section{Main results}

	\begin{theorem}\label{dys1}
	Let $X$ be a Banach space such that 
\begin{enumerate}
\item[(a)] $X^*$  is separable, and
\item[(b)] each porous subset of $X$ is $\Gamma$-null.
\end{enumerate}
 Let $G \subset X$ be an open set and $f$ an arbitrary real function on $G$.
 Then there exists a $\Gamma$-null set $A \subset G$
 such that, for each $x \in G \setminus A$, one from the following
	 assertions holds:
	\begin{enumerate}
\item  $f$ is Fr\' echet differentiable at $x$.
	\item  $f_H^{+}(x,v) = \infty$ for each $v \in X$  and  $f_H^{-}(x,v)= -\infty$ for each $v \in X$.
\item  $f_H^{+}(x,v)= \infty$ for each $v \in X$,\  $L(v): = f_H^{-}(x,v) \in X^*$
	and  $L$ is a Fr\' echet subgradient of $f$ at $x$.
\item $f_H^{-}(x,v)= -\infty$ for each $v \in X$,\  $L(v): = f_H^{+}(x,v) \in X^*$
	and  $L$ is a Fr\' echet supergradient of $f$ at $x$.	
	\end{enumerate}
		\end{theorem}
		\begin{proof}
Let $H$ be the set from Theorem I (at the end of Preliminaries), $S_1$ the set from Corollary \ref{s1}
 and $P_1$, $P_2$ the sets from Lemma \ref{hl} and Corollary \ref{hl2}. Set  $A:= H \cup S_1 \cup P_1 \cup P_2$.  Then $A$ is $\Gamma$-null by  \eqref{spdp}, \eqref{dpct} and Remark \ref{kct}(c). Let $x\in G \setminus A$ be given. Consider the conditions
$$  (C1)\ \ f_H^{+}(x,v)= \infty,\ v \in X,\ \ \ \ \ \ \text{and}\ \ \ \ \ \ (C2)\ \ f_H^{-}(x,v)= -\infty,\ v \in X.$$
 There are four possibilities.

$(i)^*$\ \ Neither (C1) nor (C2) hold. Since $x \notin P_1 \cup P_2$, we have that
$$ \liminf_{z \to x}  \frac{f(z)-f(x)}{\|z-x\|} > - \infty\ \ \ \text{and}\ \ \ 
\limsup_{z \to x}  \frac{f(z)-f(x)}{\|z-x\|} < \infty.$$ 
 So, since $x \notin S_1$, we obtain that
 $f$ is both Fr\' echet subdifferentiable at $x$ and  Fr\' echet superdifferentiable at $x$.
So  (i) holds by Fact \ref{ssdi}(ii).

$(ii)^*$\ \  Both (C1) and (C2) hold. In other words,  (ii) holds.

$(iii)^*$\ \ (C1) holds and (C2) does not hold. Since $x \notin P_1 \cup S_1$, we obtain as 
in the case $(i)^*$ that $f$ is Fr\' echet subdifferentiable at $x$. Let $L$ be a  Fr\' echet subgradient of $f$ at $x$.
Fact \ref{ssdi}(i) implies that $f_H^{-}(x,v) \geq L(v)$ for each $v \in X$. To prove that $f_H^{-}(x,v) = L(v)$ for each $v \in X$, suppose to the contrary that $f_H^{-}(x,v) > L(v)$ for some $v\in X$. Then
$$ f_H^{-}(x,v) + f_H^{-}(x,-v) >L(v)+L(-v)=0,$$
 which contradicts $x \notin H$. So  (iii) holds.

$(iv)^*$\ \ (C2) holds and (C1) does not hold. Then, quite analogously as in the case $(iii)^*$,
 we obtain that (iv) holds.
\end{proof}

\begin{remark}\label{sppr}
The assumptions (a) and (b) are satisfied (see \cite{LP} or \cite{LPT}) if
\begin{enumerate}
\item $X$ is a subspace of $c_0$, or
	\item $X=C(K)$, where $K$ is a separable compact space, or
	\item  $X$ is the Tsirelson space.
	\end{enumerate}
\end{remark}
					
\begin{theorem}\label{dys2}
	Let $X$ be a separable Banach space.  Let $G \subset X$ be an open set and $f$ an arbitrary real function on $G$.
Then there exists a  set $A \in \tilde{\C}$
 such that, for each $x \in G \setminus A$, one from the following
	 assertions holds:
	\begin{enumerate}
\item  $f$ is Hadamard differentiable at $x$.
	\item  $f_H^{+}(x,v)= \infty$ for each $v \in X$  and  $f_H^{-}(x,v)= -\infty$ for each $v \in X$.
\item  $f_H^{+}(x,v)= \infty$ for each $v \in X$ and $L(v): = f_H^{-}(x,v) \in X^*$.
\item $f_H^{-}(x,v)= -\infty$ for each $v \in X$ and  $L(v): = f_H^{+}(x,v)\in X^*$.	
	\end{enumerate}
		\end{theorem}
\begin{proof}
We obtain the proof, if we replace in the proof of Theorem \ref{dys1}:
\begin{enumerate}
\item ``Fr\' echet'' by ``Hadamard'',
\item ``$\Gamma$-null'' by  $\tilde{\C}$-null'' and
\item  ``$S_1$'' by ``$S_2$'' (where $S_2$ is the set from Corollary \ref{s2}).
 \end{enumerate}
(Of course, now we do not use Remark \ref{kct}(c) and Fact \ref{ssdi}(i).)
\end{proof}
Now we state explicitly generalizations of Stepanov theorem, which immediately follow
 from the above theorems.
\begin{corollary}\label{incn}
Let $X$, $G$ and $f$ be as in Theorem \ref{dys1}.
 Let $\tilde S(f)$ be the set of all $x\in G$, for which there exist $v_1,v_2\in S_X$
 and $\delta_1>0$, $\delta_2>0$ such that 
$$ \limsup_{y \to x, y \in C(x,v_1,\delta_1)} \frac{f(y)-f(x)}{\|y-x\|} < \infty
\ \ \text{and}\ \   \liminf_{y \to x, y \in C(x,v_2,\delta_2)} \frac{f(y)-f(x)}{\|y-x\|} > - \infty.$$
 Then there exists a $\Gamma$-null set $N\subset X$ such that $f$ is Fr\' echet differentiable at all points of $\tilde S(f)\setminus N$.
\end{corollary}  
\begin{corollary}\label{inse}
Let $X$, $G$ and $f$ be as in Theorem \ref{dys2}.
 Let $\tilde S(f)$ be the set of all $x\in G$, for which there exist $v_1,v_2\in S_X$
 and $\delta_1>0$, $\delta_2>0$ such that 
$$ \limsup_{y \to x, y \in C(x,v_1,\delta_1)} \frac{f(y)-f(x)}{\|y-x\|} < \infty
\ \ \text{and}\ \   \liminf_{y \to x, y \in C(x,v_2,\delta_2)} \frac{f(y)-f(x)}{\|y-x\|} > - \infty.$$
 Then there exists a  set $N\in \ct(X)$ such that $f$ is Hadamard differentiable at all points of $\tilde S(f)\setminus N$.
\end{corollary}
\begin{remark}\label{inco}
For $X=\R^2$, Corollary \ref{incn} was first proved by S. Saks in  \cite[Th. 7, p. 238]{Sa1}.
The ``basic version of the Denjoy-Young-Saks theorem'' of \cite{Pr} easily implies a weaker version  (in which $v_2 = -v_1$) of Corollary \ref{inse}.   
\end{remark} 
\begin{remark}\label{nejs}
The following easy examples show that Theorem \ref{dys2} gives the strongest  relation among Hadamard directional derivatives
	$f^+_H (x,v)$, $f^-_H (x,v)$ ($v \in X$) which holds a.e. for an arbitrary function on $X$.
	 Here ``a.e.'' can be taken in the sense of $\ct$-null sets, Aronszajn null sets, Haar null sets
	 or $\Gamma$-null sets. 
		The above statement has an obvious precise meaning, but we omit its formal (rather lengthy)
	 formulation.
	\smallskip
	
	Let $X$ be a separable Banach space, $L \in X^*$ and $C$, $D$ two disjoint countable sets dense
	 in $X$. 
	\begin{enumerate}
	\item
	If $f:= L$, then  $f'_H(x)=L$ for each $x \in X$.
	\item
	If $f:= \chi_C - \chi_D$, then  $f^+_H(x,v) = \infty$ and $f^-_H(x,v) = -\infty$
		 for each $x \in X \setminus (C \cup D)$ and $v \in X$ (and the set $X \setminus (C \cup D)$ is not null
	 in any sense mentioned above).
	\item
	If $f:= L + \chi_C$, then $f_H^{+}(x,v)= \infty$  and $L(v) = f_H^{-}(x,v)$ for each  $x \in X \setminus C$  and $v \in X$.
\item 
If $f:= L - \chi_C$, then $f_H^{-}(x,v)= -\infty$  and $L(v) = f_H^{+}(x,v)$ for each  $x \in X \setminus C$  and $v \in X$.
	\end{enumerate}

\end{remark}

{\bf Acknowledgments} \  The research  was partially supported by the grant GA\v CR P201/12/0436.

 \end{document}